\pdfoutput=1
\documentclass{egpubl}
\SpecialIssueSubmission
\newif\ifpaperArxiv
\paperArxivtrue
\usepackage[T1]{fontenc}
\usepackage{dfadobe}
\usepackage{cite}
\BibtexOrBiblatex
\electronicVersion
\PrintedOrElectronic
\usepackage{graphicx,egweblnk}

\usepackage{amsmath,amssymb,amsthm,mathtools,cancel}
\usepackage{microtype,booktabs,array,enumitem,xcolor,listings,needspace}
\newif\ifshownotes
\shownotesfalse

\usepackage{bookmark}
\hypersetup{unicode=true,colorlinks=true,linkcolor=blue!45!black,citecolor=blue!45!black,urlcolor=blue!45!black,pdfsubject={Knupp's conjecture and quartic reduction of boundary Jacobian minimization}}
\setlist[itemize]{topsep=3pt,itemsep=2pt,parsep=0pt,leftmargin=1.5em}
\setlist[enumerate]{topsep=3pt,itemsep=3pt,parsep=0pt,leftmargin=1.8em}
\allowdisplaybreaks[1]
\newtheoremstyle{exposition}{6pt}{6pt}{\normalfont}{}{\bfseries}{.}{.5em}{}
\theoremstyle{exposition}
\newtheorem{theorem}{Theorem}

\newtheorem{lemma}{Lemma}
\newcommand{\restatedlemmaname}{}
\newtheorem*{restatedlemma}{\restatedlemmaname}
\newtheorem{sublemma}{Sub-lemma}[lemma]

\newtheorem{derivation}{Derivation}

\makeatletter
\patchcmd{\SFB@ssect}{\edef}{\protected@edef}{}{}
\patchcmd{\SFB@ssect}{\edef}{\protected@edef}{}{}
\patchcmd{\SFB@ssect}{\edef}{\protected@edef}{}{}
\patchcmd{\SFB@ssect}{\edef}{\protected@edef}{}{}
\patchcmd{\SFB@ssect}{\edef}{\protected@edef}{}{}
\renewcommand{\tableofcontents}{%
\section*{Contents}%
\begingroup\small\@starttoc{toc}\endgroup\medskip}
\EGlocalpagenumber
\patchcmd{\@oddhead}{\@shortauthor\ / }{}{}{}
\patchcmd{\@evenhead}{\@shortauthor\ / }{}{}{}
\makeatother
\title{Validating Hexahedra through their Boundaries}
\author[]{Paul Zhang}

\ifpaperArxiv
\hypersetup{pdfkeywords={}}
\copyrightTextTitPag{}
\copyrightTextRunPag{}
\makeatletter
\def\ps@titlepage{\let\@mkboth\@gobbletwo
\def\@oddhead{}%
\def\@oddfoot{}%
\let\@evenhead=\@oddhead
\let\@evenfoot=\@oddfoot
\let\sectionmark=\EmptySectionmark
\let\subsectionmark=\EmptySubsectionmark
}
\makeatother
\else
\copyrightTextTitPag{Draft prepared using the SGP 2026 template.}
\copyrightTextRunPag{Draft prepared using the SGP 2026 template.}
\fi

\begin{document}
\maketitle
\begin{abstract}
Trilinear hexahedral mesh elements are used in finite element analysis to simulate volumetric phenomena. Realism of the simulation mandates that mesh elements through the trilinear map maintain positive volume, or mathematically, that the trilinear map maintains positive Jacobian determinant (jacdet).
While jacdet positivity generally needs to be verified in the full element volume, we prove
Knupp's conjecture which posits that for trilinear hexahedra, jacdet positivity on the boundary of the element is sufficient to guarantee positivity in the entire element.
We further prove that for any valid hexahedron, its globally minimal jacdet must reside on the boundary. Lastly, we prove that the globally minimal jacdet on the boundary of a hexahedron, valid or not, must be achieved within a finite set of candidate points that can be determined by quartic root finding. Combining these results, we can algorithmically validate a hexahedron through evaluation of its jacdet on a finite set of boundary points.

\end{abstract}

\section{Introduction}\label{sec:introduction}
Hexahedral meshes discretize three-dimensional domains for numerical analysis and physical simulation, including finite element analysis (FEA)~\cite{ciarlet2002,zlamal1968}.
Hexahedral meshes also support the construction of volumetric spline spaces for isogeometric analysis (IGA)~\cite{hughes2005,wei2018}.
Generating and processing such meshes remains an active research area.
For a survey on the field of hexahedral meshing, see~\cite{pietroni2023}.
We focus on the problem of determining validity of individual trilinear hexahedral elements.

Element validity is necessary for these meshes to serve as simulation domains. A trilinear element maps a reference cube into physical space, and the Jacobian determinant (jacdet) measures the signed local volume change. With the reference orientation fixed, a zero determinant indicates local collapse, while a negative determinant indicates local inversion. In a deforming solid, both violate the requirement that material retain a locally nonsingular, orientation-preserving configuration. The validity criterion is therefore strict positivity of the determinant throughout the reference cube. Unfortunately, corner samples, or positivity along all twelve edges, do not suffice to establish this condition~\cite{knupp1990}.

Existing certification methods establish positivity over the whole element volumetrically, using, for example, recursive bounding~\cite{johnen2013,johnen2017,knabner2003} or sum-of-squares polynomial optimization~\cite{marschner2020}.
While these methods invariably treat the problem as volumetric, we can take inspiration from the planar bilinear quadrilateral case. For a planar bilinear quadrilateral, positive jacdet at the four corners of the bilinear map guarantees positivity throughout the element~\cite[p.~315, Eq.~(22)]{knupp1990}. This remarkable result simplifies the problem of planar bilinear quad validity into jacdet evaluation over just a finite set of boundary points. Following this line of thought, it would be equally remarkable if a purely boundary-based criterion can be established for the trilinear hexahedron. Knupp conjectures exactly this with the face-test statement~\cite[p.~322]{knupp1990}:
\begin{quote}
``if \(J\) is positive on each of the six faces of the hexahedron, then \(J\) is positive everywhere in the interior.''
\end{quote}
We resolve this conjecture and derive an algorithm for testing trilinear hex validity through evaluation of the jacdet at a finite set of boundary points. To achieve this outcome we take the following steps: 
\begin{itemize}
\item prove Knupp's conjecture establishing that strict positivity of the jacdet on the six faces implies strict positivity throughout the element.
\item prove the stronger condition that, for a valid trilinear hexahedron, the global minimum jacdet is attained on the boundary.
\item prove that the boundary minimum, for a valid or invalid element, is attained at a finite set of candidate points derived from roots of polynomials of degree at most four.
\item combine these results into a boundary-based hexahedron validity algorithm
\end{itemize}

\section{Related work}\label{sec:related-work}

\paragraph*{Hexahedral mesh generation and analysis.}
Reference-domain maps have long been used to generate finite element meshes, including through blending-function interpolation~\cite{gordon1973}. Modern hex-generation approaches include grid-based methods, polycube constructions, and frame-field-guided parameterizations. Examples include interactive cuboid decomposition~\cite{li2021} and volume quantization that allows the singularity structure to simplify during mesh generation~\cite{bruckler2026}. Hex meshes also provide domains for spline constructions used in IGA~\cite{wei2018}. For a survey on these approaches together with connectivity editing and mesh optimization, see~\cite{pietroni2023}.

\paragraph*{Validity tests, Jacobian bounds, and subdivision.}
\cite{knupp1990}'s analysis of isoparametric-map invertibility demonstrates the failure of both corner and edge positivity tests and proposes the face-test conjecture. A subsequent positivity test recursively reduces a three-variable Jacobian-positivity problem to two-variable and then one-variable polynomial tests, using piecewise-linear tangent bounds on quadratics~\cite{knabner2003}.
Increasing the tangent resolution strengthens these sufficient conditions. The test certifies positivity rather than computing the minimum jacdet.

Bernstein--B\'ezier methods likewise certify positivity through bounds, but refine those bounds by adaptively subdividing the reference domain~\cite{johnen2013}. The same work computes the minimum for quadratic triangles directly by checking stationary points and the boundary. A specialization to linear hexahedra exploits the structure of the determinant polynomial to construct B\'ezier coefficients efficiently and refine the bounds when needed~\cite{johnen2017}. 

\paragraph*{Polynomial optimization and mesh repair.}
Jacobian bounds can guide vertex updates that untangle curvilinear meshes~\cite{toulorge2013}. For trilinear hexes, sum-of-squares (SOS) and moment relaxations formulate Jacobian minimization over the reference cube as polynomial optimization and use the resulting minimum locations in mesh repair~\cite{marschner2020}. These relaxations belong to the broader theory of polynomial positivity and moment optimization~\cite{putinar1993,lasserre2001}. Subsequent geometry-processing work applies SOS to queries on higher-order primitives~\cite{marschner2021}; follow-up work accelerates SOS collision detection for curved shapes and paths~\cite{zhang2023}. Such acceleration techniques may also inform future SOS-based validation and repair methods. Our result supplies a validity test and boundary minimum for a fixed trilinear element; a mesh-repair update is a separate task.

\paragraph*{Low-degree polynomial root finding.} %
Our approach leads to finding roots of polynomial equations of degree at most four, whose solution is classical~\cite[\S1.11(iii)]{dlmf}.
Efficient numerical quartic solvers are also available~\cite{orellana2020}.
We use these methods in our final algorithm. 

\Needspace{8\baselineskip}
\section{Preliminaries}\label{sec:preliminaries}\label{sec-setup-title}

Let \(Q=[0,1]^3\), with reference coordinates \((x,y,z)\).
Let
\begin{equation}
F(x,y,z):Q\longrightarrow\mathbb R^3\label{web-equation-01}
\end{equation}
be the map of a trilinear hexahedron; It is affine in each reference coordinate \((x,y,z)\) separately.
The image \(\mathcal H:=F(Q)\) is the hexahedral mesh element.
Define the derivative vectors
\begin{equation}
U=\frac{\partial F}{\partial x},\qquad
V=\frac{\partial F}{\partial y},\qquad
W=\frac{\partial F}{\partial z}.
\label{eq:derivative-vectors}
\end{equation}
The Jacobian matrix and its determinant are
\begin{equation}
DF=\left[U\mid V\mid W\right],\qquad J=\det(DF),\label{web-equation-03}
\end{equation}
where the brackets denote horizontal concatenation of column vectors.
We say that the hexahedron \(\mathcal H\), with map \(F\), is \emph{valid} if and only if \(J(p)>0\;\forall p\in Q\).

\section{Positivity and minimization of the Jacobian determinant}\label{sec:theoretical-results}
We start by formalizing the conjecture as \autoref{thm:knupp-conjecture}.
\begin{theorem}[Knupp's conjecture]\label{thm:knupp-conjecture}
Following the notation of \autoref{sec:preliminaries},
if \(J>0\) on \(\partial Q\), then
\begin{equation}
J(p)>0\qquad \forall p\in Q.
\label{eq:main-theorem}
\end{equation}
Thus hexahedron \(\mathcal H\) is valid.
\end{theorem}
Note that the condition of \autoref{thm:knupp-conjecture} is \emph{strict} boundary positivity. We first show by counterexample in \autoref{sec:non-strict-counterexample} that a non-strict variant of this theorem is false. We then prove the strict statement \autoref{thm:knupp-conjecture} with the following strategy. We will presume in \autoref{sec-positivity-title} that there are locations in \(Q\) with \(J=0\). We will use the geometry of this zero set to obtain \autoref{opposite-derivative-signs}, then derive the opposite inequality in \autoref{same-derivative-sign} to arrive at a contradiction. Since \(J\) is strictly positive on the boundary, and has no zeros on the interior, by continuity, it must remain strictly positive.

\subsection{Counterexample to non-strict boundary positivity}\label{sec:non-strict-counterexample}
Strict boundary positivity in Knupp's conjecture cannot be weakened to nonnegativity. On \(Q=[0,1]^3\), consider the trilinear map
\begin{equation}
F(x,y,z)=
\begin{pmatrix}
x(1-y-z)\\
y(1-x-z)\\
z(1-x-y)
\end{pmatrix}.
\label{eq:non-strict-map}
\end{equation}
Direct differentiation gives
\begin{equation}
J(x,y,z)=\det(DF(x,y,z))=(1-x-y-z)^2-4xyz.
\label{eq:non-strict-jacobian}
\end{equation}
On the two faces perpendicular to the \(x\)-axis,
\begin{equation}
\begin{aligned}
J(0,y,z)&=(1-y-z)^2\geq0,\\
J(1,y,z)&=(y-z)^2\geq0.
\end{aligned}
\label{eq:non-strict-faces}
\end{equation}
Nonnegativity on the other four faces follows by permuting \(x,y,z\), which $J$ is invariant to. Nevertheless, at the cube's center,
\begin{equation}
J(1/2,1/2,1/2)=-1/4<0.
\label{eq:non-strict-center}
\end{equation}
Thus \(J\geq0\) on \(\partial Q\) does not imply \(J\geq0\) in \(Q\).

\subsection{Proof of Knupp's conjecture}\label{sec-positivity-title}\label{sec:knupp-proof}
We begin our proof of \autoref{thm:knupp-conjecture} here. 
Our strategy will be to rule out every zero of \(J\) in \(Q\). This is enough to prove positivity: if \(J\) were negative at an interior point, continuity along a line segment from that point to the boundary would force \(J\) to vanish somewhere along that segment.

Suppose, for a contradiction, that the set
$Z=\{p\in Q:J(p)=0\}$
is nonempty. Boundary positivity guarantees that every point of \(Z\) lies in the strict interior of \(Q\).
We establish some regularity conditions on \(Z\) to make it easier to reason about.

\begin{lemma}[{Each connected component of \(Z\) is a compact smooth surface enclosing a bounded region}]\label{lemma-1}

Under the boundary-positivity assumption, \(p\in Z \rightarrow \nabla J(p)\ne0\). Consequently, each connected component \(S\) of \(Z\) is a compact, smooth, two-dimensional embedded surface without boundary. The surface \(S\) is the boundary of a bounded three-dimensional region \(\Omega\subset\mathbb R^3\).

\end{lemma}

\noindent\emph{The proof of \autoref{lemma-1} is in \autoref{proof-1}.} The crux is to rule out a point \(p\in Z\) with \(\nabla J(p)=0\). It turns out that the algebraic conditions \(J(p)=0\) and \(\nabla J(p)=0\) are strict enough to force at least one of the three coordinate squares through \(p\) in \(Q\) to have a planar image under \(F\). We then analyze coordinate planes with planar images and show that, similar to the bilinear quad case, positivity of \(J\) on its perimeter implies positivity throughout the square. That perimeter lies on \(\partial Q\), where \(J>0\), so \(J(p)>0\), contradicting \(J(p)=0\). Thus \(\nabla J(p)\ne0\) at every \(p\in Z\).

Let \(S\) be one connected component of \(Z\), and let \(p_-\) and \(p_+\) minimize and maximize \(x\) on \(S\). These extreme points exist by compactness. In reference-coordinate space, the outward normals at \(p_-\) and \(p_+\) point in the negative and positive \(x\)-directions, respectively. Since \(J=0\) on \(S\), \(\nabla J\) is normal to \(S\), so
\begin{equation}
\begin{aligned}
\frac{\partial J}{\partial y}(p_-)&=\frac{\partial J}{\partial z}(p_-)=0,\\
\frac{\partial J}{\partial y}(p_+)&=\frac{\partial J}{\partial z}(p_+)=0.
\end{aligned}\label{extreme-partials}
\end{equation}
The continuous, nonzero normal field \(\nabla J\) must point either outward everywhere on the connected surface \(S\), or inward everywhere; it cannot switch without vanishing. Thus \(\nabla J(p_-)\) and \(\nabla J(p_+)\) point in opposite \(x\)-directions, giving
\begin{equation}
\boxed{\frac{\partial J}{\partial x}(p_-)\frac{\partial J}{\partial x}(p_+)<0.}\label{opposite-derivative-signs}
\end{equation}

So far, the argument has used the geometry of an enclosed regular zero surface. The following lemma uses properties of trilinearity to show that the \textbf{derivatives \(\frac{\partial J}{\partial x}(p_-)\) and \(\frac{\partial J}{\partial x}(p_+)\) must have the same nonzero sign}.

\begin{lemma}[{On \(S\), \(\frac{\partial J}{\partial x}\) has one nonzero sign wherever \(\frac{\partial J}{\partial y}=\frac{\partial J}{\partial z}=0\)}]\label{lemma-2}

Under the boundary-positivity assumption, let \(S\) be a connected component of \(Z\). Suppose \(p,q\in S\) satisfy

\begin{equation}
\begin{aligned}
\frac{\partial J}{\partial y}(p)&=\frac{\partial J}{\partial z}(p)=0,\\
\frac{\partial J}{\partial y}(q)&=\frac{\partial J}{\partial z}(q)=0.
\end{aligned}\label{web-equation-37}
\end{equation}

\Needspace{4\baselineskip}
Then

\begin{equation}
\boxed{\frac{\partial J}{\partial x}(p)\frac{\partial J}{\partial x}(q)>0.}\label{same-derivative-sign}
\end{equation}

\end{lemma}

\noindent\emph{The proof of \autoref{lemma-2} is in \autoref{proof-2}.} The strategy is to express \(\frac{\partial J}{\partial x}\), at points where \(\frac{\partial J}{\partial y}=\frac{\partial J}{\partial z}=0\), as a product of continuous scalar functions defined throughout \(S\). We obtain this factorization by expanding the derivatives of the determinant and using the column dependence implied by \(J=0\). We then show that none of the factors can vanish: a vanishing factor would force a coordinate square through that point to have a planar image under \(F\), contradicting \(J=0\) by the same perimeter-positivity argument used in \autoref{lemma-1}. Since \(S\) is connected, each continuous, nonzero factor has a fixed sign. Their product therefore has the same nonzero sign at every point satisfying \autoref{web-equation-37}.

\autoref{extreme-partials} supplies precisely the conditions of \autoref{lemma-2} for \(p_-\) and \(p_+\). \autoref{same-derivative-sign} in \autoref{lemma-2} therefore contradicts \autoref{opposite-derivative-signs}.
Consequently, \(Z\) is empty. The continuity argument at the beginning now gives
\begin{equation}
\boxed{J>0\quad\text{throughout }Q.}\label{web-equation-49}
\end{equation}
This proves \autoref{thm:knupp-conjecture}.

\subsection{Boundary attainment of the minimum Jacobian}\label{sec-minimum-title}

While \autoref{sec:knupp-proof} establishes that hexahedron validity can be verified purely by knowing minimal boundary jacdet, it does not guarantee that the minimal boundary jacdet is the minimal jacdet of the entire hexahedron. Nonetheless, \cite{marschner2020} found in experiments that the minimal jacdet was always achieved on the boundary. We know from \autoref{sec:non-strict-counterexample} that this cannot generally extend to all hexahedra. However we prove in this subsection that, conditional on hexahedron validity, its minimal jacdet is also attained on its boundary.

\begin{theorem}[Boundary minimum]\label{thm:boundary-minimum}
If \(J>0\) on \(\partial Q\), then
\begin{equation}
\min_{r\in Q}J(r)=\min_{r\in\partial Q}J(r)>0.
\label{eq:boundary-minimum}
\end{equation}
\end{theorem}
\noindent\emph{The proof of \autoref{thm:boundary-minimum} is in \autoref{proof:boundary-minimum}; here we give an outline.}
To establish the minimum claim, we must also rule out an interior minimum whose value is smaller than every boundary value. To do that, we suppose \(p\in(0,1)^3\) is a strictly interior local minimizer with \(J(p)>0\) and show that \(J\) is constant along at least one coordinate-line segment through \(p\) extending to \(\partial Q\).

Along each coordinate line, \(J\) is a polynomial of degree at most two. At an interior local minimum, all three linear coefficients vanish and all three quadratic coefficients \(q_i\) are nonnegative. The key algebraic step uses determinant identities to show that \(J(p)q_i\le0\) for at least one coordinate direction. Because the hexahedron is valid by \autoref{thm:knupp-conjecture}, \(J(p)>0\), forcing that quadratic coefficient to vanish. On that coordinate line, \(J\) is therefore constant, so its boundary endpoints have the same value as \(p\). Applying this to an interior global minimizer gives
\begin{equation}
\boxed{\min_{p\in Q}J(p)=\min_{p\in\partial Q}J(p)>0.}\label{web-equation-67}
\end{equation}

\subsection{Face minimization via quartic root finding}\label{sec:face-minimum}

We have established in \autoref{thm:boundary-minimum} that the minimum Jacobian determinant of a valid hexahedron is attained on its boundary. We therefore turn to determining this boundary minimum. The face-minimization method below applies to both valid and invalid hexahedra, although only for valid hexahedra is the boundary minimum guaranteed to equal the minimum over the entire cube.

We will construct a finite set of candidate points $\mathcal C$ guaranteed to include a global minimizer of the jacdet on a face. We first determine the points to check if a global minimizer lies on the face perimeter (edges of the cube). We then show that any smaller minimum in the face interior can be found by solving a quartic equation, which admits closed form solution. Combining the candidates guarantees that at least one attains the global face minimum. Taking the minimum over all six faces gives the minimum jacdet on the boundary of the hexahedron.

We first establish notation for the face calculation. On a face \(x=x_0\), with \(x_0\in\{0,1\}\), define
\begin{equation}
K=[0,1]^2,\qquad f(s,t)=J(x_0,s,t).\label{face-face-definition}
\end{equation}
The same construction applies with \(y\) or \(z\) fixed.
Grouping by powers of \(s\) separates the quadratic dependence on \(s\) from the coefficient functions of \(t\):
\begin{equation}
\begin{aligned}
f(s,t)&=a(t)s^2+b(t)s+c(t),\\
a(t)&=a_0+a_1t,\\
b(t)&=b_0+b_1t+b_2t^2,\\
c(t)&=c_0+c_1t+c_2t^2.
\end{aligned}\label{face-face-form}
\end{equation}
The eight real coefficients are determined by the hexahedron and the chosen face. Note the \(s^2t^2\) coefficient is 0. This form is easily verified using symbolic manipulation tools. 

\subsubsection{Find the minimum on the four edges first}\label{sec:face-edge-minimum}

Parameterize each edge by \(r\in[0,1]\). The four restricted polynomials are

\begin{equation}
\begin{array}{c|c}
\text{point on the edge}&\text{value of }f\\\hline
(0,r)&c(r)\\
(1,r)&a(r)+b(r)+c(r)\\
(r,0)&a(0)r^2+b(0)r+c(0)\\
(r,1)&a(1)r^2+b(1)r+c(1).
\end{array}\label{face-edge-restrictions}
\end{equation}

Each restriction is a quadratic \(q(r)=d_2r^2+d_1r+d_0\), whose minimum on \([0,1]\) occurs at an endpoint or at the vertex of an upward-opening parabola. The perimeter candidates in \(\mathcal C\) are therefore the four corners of \(K\) and the point on each edge with \(d_2>0\) corresponding to
\begin{equation}
r_*=-\frac{d_1}{2d_2},\qquad 0<r_*<1.\label{face-edge-vertex}
\end{equation}
At least one of these at most eight perimeter candidates attains the face perimeter minimum.

\subsubsection{Reduce the interior calculation to a quartic}

The edge candidates are sufficient unless the face has a smaller value in its interior. At an interior minimizer, both partial derivatives of \(f\) vanish. We will solve the equation in \(s\) explicitly, then use the equation in \(t\) to locate the remaining candidates.

\begin{equation}
0=\frac{\partial f}{\partial s}(s,t)=2a(t)s+b(t).\label{face-stationary-s}
\end{equation}
\begin{equation}
0=\frac{\partial f}{\partial t}(s,t)=a'(t)s^2+b'(t)s+c'(t).\label{face-stationary-t}
\end{equation}

Before dividing by \(a(t)\), however, we need to account for the possibility that this coefficient is zero.
At a strictly interior local minimizer \((s_0,t_0)\), the second-derivative condition gives \(2a(t_0)=\frac{\partial^2 f}{\partial s^2}(s_0,t_0)\ge0\). If \(a(t_0)=0\), \autoref{face-stationary-s} gives \(b(t_0)=0\), so
\begin{equation}
f(s,t_0)=\underbrace{a(t_0)}_{0}s^2+\underbrace{b(t_0)}_{0}s+c(t_0)=c(t_0).\label{face-flat-slice}
\end{equation}
The same value is therefore attained at the face-perimeter points \((0,t_0)\) and \((1,t_0)\), so an interior minimum smaller than the face perimeter minimum must have \(a(t)>0\). At such a point, \autoref{face-stationary-s} gives
\begin{equation}
s=s_*(t):=-\frac{b(t)}{2a(t)}.\label{face-recovered-coordinate}
\end{equation}
Substitute this value of \(s\) into \autoref{face-stationary-t} and multiply by \(4a(t)^2\) to remove the denominators:
\begin{equation}
\begin{gathered}
\frac{\partial f}{\partial t}(s_*(t),t)=0=a'(t)\left(-\frac{b(t)}{2a(t)}\right)^2+b'(t)\left(-\frac{b(t)}{2a(t)}\right)+c'(t),\\[4pt]
0=4a(t)^2c'(t)-2a(t)b(t)b'(t)+a'(t)b(t)^2=:P(t).
\end{gathered}\label{face-quartic-elimination}
\end{equation}
The detailed expansion in \autoref{app:quartic-expansion} shows that \(P\) has degree at most four.

If \(P\) is not the zero polynomial, compute its roots using the classical formulas for polynomials of degree at most four \cite[\S1.11(iii)]{dlmf}, and keep the distinct real roots in \((0,1)\).
For each root, retain the point \((s_*(t),t)\) only when \(a(t)>0\) and \(0<s_*(t)<1\). There are at most four retained interior points. These are the interior candidates in \(\mathcal C\).
Some retained stationary points can be saddles. We can simply keep them: every retained point is a feasible point of the square, so including extra points cannot produce a value below the true global minimum.

\subsubsection{If the quartic vanishes identically, the edges suffice}

A root finder needs a nonzero polynomial. If every coefficient of \(P\) is zero, then every \(t\) solves \(P(t)=0\); enumerating roots is impossible. We need a separate conclusion for this case.

\Needspace{6\baselineskip}
\begin{lemma}[{\(P\equiv0\) implies a global face minimizer on the face perimeter}]\label{face-lemma-zero}

The algorithm can return the smallest edge-candidate value without adding any interior candidates.

\end{lemma}

\begin{proof}[Proof of \autoref{face-lemma-zero}]

Choose a global minimizer \((s_0,t_0)\). If this point lies on the face perimeter, the claim holds. For a strictly interior minimizer, the second-derivative condition gives \(a(t_0)\ge0\); if \(a(t_0)=0\), \autoref{face-flat-slice} shows that the same value is attained on the face perimeter.

Now consider the case where \((s_0,t_0)\) is strictly interior and \(a(t_0)>0\). \autoref{face-recovered-coordinate} gives \(s_0=s_*(t_0)\in(0,1)\). We will show that the value at this minimizer is also attained on the face perimeter.

We will follow the parabola's vertex as \(t\) changes, show that its value stays constant, and find a face-perimeter point with that value. Where \(a(t)>0\), the value at the vertex is
\begin{equation}
h(t):=f(s_*(t),t).\label{face-vertex-value}
\end{equation}
The chain rule gives
\begin{equation}
h'(t)=\underbrace{\frac{\partial f}{\partial s}(s_*(t),t)}_{=0}\,s_*'(t)
+\underbrace{\frac{\partial f}{\partial t}(s_*(t),t)}_{=0}=0.\label{face-vertex-value-derivative}
\end{equation}
Stationarity in \(s\) makes the first term zero. By \autoref{face-quartic-elimination}, the second term is \(P(t)/(4a(t)^2)=0\) because \(P\equiv0\).
Thus, as we follow the vertex from \(t_0\) while \(a(t)>0\), its value stays equal to the global minimum \(h(t_0)\).

To find a face-perimeter point with that value, decrease \(t\) toward \(t=0\), with dependent \(s=s_*(t)\) staying at the parabola's vertex. While \(a(t)>0\) and the vertex remains strictly inside the face, \(s_*(t)\) is continuous, so we can continue decreasing \(t\). If we reach \(t=0\), \(s=0\), or \(s=1\) while \(a(t)>0\), continuity gives the face-perimeter point the same minimum value.

The remaining case is that \(a(t)\) reaches zero before or at the same time as we reach the face perimeter. Let \(t_1\) be that value of \(t\), so \(a(t_1)=0\). As \(t\) approaches \(t_1\), the bound \(0<s_*(t)<1\) gives
\begin{equation}
\begin{aligned}
b(t)&=-2a(t)s_*(t)\longrightarrow0,\\
c(t)&=h(t_0)+a(t)s_*(t)^2\longrightarrow h(t_0).
\end{aligned}\label{face-zero-endpoint-limit}
\end{equation}
By continuity, \(b(t_1)=0\) and \(c(t_1)=h(t_0)\). Therefore
\begin{equation}
\begin{aligned}
f(s,t_1)&=\underbrace{a(t_1)}_{0}s^2+\underbrace{b(t_1)}_{0}s+c(t_1)\\
&=h(t_0)\qquad\forall s\in[0,1].
\end{aligned}\label{face-flat-endpoint-value}
\end{equation}
In particular, the face-perimeter points \((0,t_1)\) and \((1,t_1)\) attain the same minimum. Thus, in all cases, when \(P\equiv0\), the global minimum is attained on the face perimeter.

\end{proof}

For example, \(f(s,t)=(s-t)^2\) has a whole line of minimizers. Its polynomial \(P\) is identically zero, and the same minimum is attained at the corners \((0,0)\) and \((1,1)\). \autoref{face-lemma-zero} handles such cases without trying to enumerate infinitely many stationary points.

\subsubsection{Finite candidate set}

The edge candidates and the retained points recovered from the roots of \(P\) contain a global face minimizer; when \(P\equiv0\), the edge candidates alone suffice. Finding the minimum jacdet on the boundary of a trilinear hexahedron therefore reduces to quadratic edge checks and at most one quartic equation per face. The remaining polynomial equations have degree at most four and therefore admit closed-form solutions \cite[\S1.11(iii)]{dlmf}.

\section{Algorithm}\label{sec:algorithm}

We collect the candidate locations from \autoref{sec:face-minimum} into \autoref{alg:face-minimum} to find the minimum Jacobian determinant on the cube boundary. Shared corners and edges are processed once, giving eight vertex checks and twelve quadratic edge checks. To avoid unnecessary face quartic solves, we first test a Bernstein lower bound on each face.

\subsection{Bernstein face bounds}\label{sec:bernstein-face-bounds}
Inspired by the Bernstein coefficient bounds used in~\cite[\S3.1]{johnen2017}, we express the face polynomial in the quadratic Bernstein basis:
\begin{equation}
B_0(r)=(1-r)^2,\quad B_1(r)=2r(1-r),\quad B_2(r)=r^2.
\label{alg-bernstein-functions}
\end{equation}
These functions span all quadratics since \(1=B_0+B_1+B_2\), \(r=B_1/2+B_2\), and \(r^2=B_2\). Thus every biquadratic face polynomial has the representation
\begin{equation}
f(s,t)=\sum_{i=0}^2\sum_{j=0}^2\beta_{ij}B_i(s)B_j(t).
\label{alg-face-bernstein}
\end{equation}
On \([0,1]^2\), the weights \(B_i(s)B_j(t)\) are nonnegative and sum to one, so every value of \(f\) is a weighted average of the nine coefficients. Therefore, for a face \(R\),
\begin{equation}
L_R:=\min_{0\le i,j\le2}\beta_{ij}
\ \le\ \min_{(s,t)\in[0,1]^2}f(s,t).
\label{alg-face-lower-bound}
\end{equation}
The coefficients follow from a fixed change of basis: if \(f(s,t)=\sum_{i,j=0}^2D_{ij}s^it^j\), then
\begin{equation}
\beta=TDT^{\mathsf T},\qquad
T=\begin{pmatrix}1&0&0\\1&1/2&0\\1&1&1\end{pmatrix}.
\label{alg-bernstein-conversion}
\end{equation}
Let \(M\) be the smallest determinant value already found at an evaluated boundary point. The bound gives two pruning rules depending on if we want to find the minimum boundary Jacobian determinant or if we just want to determine validity of the hexahedron:
\begin{equation}
\begin{array}{ll}
\text{minimum:} & L_R\ge M\ \Longrightarrow\ \text{skip face }R,\\[2pt]
\text{validity:} & L_R>0\ \Longrightarrow\ \text{skip face }R.
\end{array}\label{alg-face-skip-rules}
\end{equation}
A skipped face cannot improve \(M\) in a minimum query, or is already positive in a validity query. Conversely, \(L_R\le0\) does not prove invalidity: an inconclusive bound requires the quartic calculation. We use one bound per face, without subdivision.

Our implementation usees twenty determinant samples at the eight cube corners and twelve edge midpoints to obtain the cube Bernstein coefficients~\cite[\S4]{johnen2017}; each face selects nine coefficients. Because the \(s^2t^2\) term vanishes in \autoref{face-face-form}, eight perimeter samples determine each face polynomial, with no face-center evaluation.

\subsection{Pseudocode}\label{sec:boundary-algorithm}
\autoref{alg:face-minimum} combines the edge and face candidates with the minimum-pruning rule in \autoref{alg-face-skip-rules}.

\noindent\begin{minipage}{\linewidth}
\begingroup
\renewcommand{\lstlistingname}{Algorithm}
\begin{lstlisting}[language=Python,escapeinside={(*@}{@*)},caption={Boundary minimum with Bernstein face pruning.},label={alg:face-minimum}]
def boundary_minimum(F):
    J = det(DF)
    # Evaluate J at all eight cube corners.
    p = argmin(J(v) for v in cube_vertices)
    M = J(p)  # current min jacdet
    
    # Minimize J per edge
    for (v0, v1) in cube_edge_endpoints: 
        E(r) = (1-r)*v0 + r*v1  # r in [0,1]
        # J(E(r)) = d2*r^2 + d1*r + d0
        # (*@\autoref{face-edge-restrictions}@*)
        d2, d1, d0 = coefficients_in_r(J(E(r)))
        if d2 > 0:
            r_star = -d1 / (2*d2)  # (*@\autoref{face-edge-vertex}@*)
            if 0 < r_star < 1:
                u = E(r_star)
                if J(u) < M:
                    p, M = u, J(u)
    
    # Minimize J per face
    for R in the six reference-cube faces:
        f = restrict_to_face(J, R)
        # Bernstein bound from J coefficients.
        D = power_coefficient_matrix(f)
        beta = T @ D @ T.T  # (*@\autoref{alg-bernstein-conversion}@*)
        L_R = min(beta)  # (*@\autoref{alg-face-lower-bound}@*)
        if L_R >= M: # this face does not have the min
            continue  
        # (*@\autoref{face-face-form}@*)
        a, b, c = coefficients_in_s(f)
        construct P from (*@\autoref{face-quartic-elimination}@*)
        if P is identically zero:
            continue  # (*@\autoref{face-lemma-zero}@*)
        for t0 in compute_roots(P):
            if 0 < t0 < 1 and a(t0) > 0:
                # (*@\autoref{face-recovered-coordinate}@*)
                s0 = -b(t0)/(2*a(t0))
                if 0 < s0 < 1 and f(s0,t0) < M:
                    p = lift_to_cube((s0,t0), R)
                    M = f(s0,t0)
    return p, M
\end{lstlisting}
\endgroup
\end{minipage}

The routine \texttt{compute\_roots(P)} returns all distinct real roots of a nonzero polynomial of degree at most four, using an existing polynomial solver~\cite[\S1.11(iii)]{dlmf}. The face calculation retains at most four interior candidates; if \(P\equiv0\), the edge candidates already suffice by \autoref{face-lemma-zero}. The returned value is the boundary minimum and, when positive, also the volume minimum by \autoref{thm:boundary-minimum}.

For a strict-validity query, \autoref{thm:knupp-conjecture} reduces the test to positivity on the boundary. Return false as soon as an evaluated boundary value is nonpositive, and skip a face whenever \(L_R>0\) by \autoref{alg-face-skip-rules}. Unlike the minimum query, the validity test stops at the first witness of invalidity and needs no further minimization.

\section{Results}\label{sec:results}
\newcommand{\SweepHexes}{2,100,000}
\newcommand{\SweepCutoffs}{127}
\newcommand{\SweepCutoffInvalid}{63}
\newcommand{\SweepFalseNegatives}{64}

\begin{table*}[t]
\centering
\caption{Strict-validity runtimes in nanoseconds per hex (lower is faster),
with 100,000 hexes per $\sigma$. HXT~\cite{hxtsource} uses its released, unchanged default
recursion guard \texttt{if (depth > 4) return false;}. Every depth cutoff
therefore returns invalid, even without a nonpositive witness.
False negatives are valid hexes rejected by HXT, independently confirmed
by further Bernstein subdivision. \textbf{Roots+B} adds only one Bernstein face-bound
pass to \textbf{Roots}. All cutoff cases remain in the timed populations.}
\label{tab:validity-throughput}
\begingroup\small
\renewcommand{\arraystretch}{1.10}
\begin{tabular*}{\textwidth}{@{\extracolsep{\fill}}rrrrrrr@{}}
\toprule
$\sigma$ & Invalid (\%) & \shortstack{HXT\\(ns/hex)} &
\shortstack{\textbf{Roots}\\(ns/hex)} & \shortstack{\textbf{Roots+B}\\(ns/hex)} &
\shortstack{HXT depth\\cutoffs} & \shortstack{HXT false\\negatives} \\
\midrule
0.0 & 0.000 & 138.7 & 214.4 & 155.4 & 0 & 0 \\
0.1 & 0.001 & 141.2 & 1321.7 & 210.5 & 0 & 0 \\
0.2 & 6.675 & 138.0 & 1259.3 & 220.3 & 0 & 0 \\
0.3 & 44.225 & 123.0 & 792.9 & 175.1 & 7 & 5 \\
0.4 & 75.045 & 116.6 & 392.8 & 123.8 & 15 & 9 \\
0.5 & 88.707 & 104.0 & 209.3 & 92.1 & 15 & 7 \\
0.6 & 94.546 & 91.1 & 125.7 & 72.1 & 18 & 10 \\
0.7 & 97.116 & 80.9 & 87.1 & 60.2 & 14 & 9 \\
0.8 & 98.249 & 71.9 & 69.3 & 53.8 & 7 & 3 \\
0.9 & 98.860 & 66.9 & 58.7 & 49.7 & 6 & 5 \\
1.0 & 99.208 & 64.2 & 52.5 & 46.4 & 5 & 1 \\
1.1 & 99.419 & 62.7 & 48.8 & 44.7 & 4 & 2 \\
1.2 & 99.549 & 61.2 & 45.5 & 43.4 & 6 & 2 \\
1.3 & 99.637 & 58.2 & 44.2 & 42.6 & 5 & 3 \\
1.4 & 99.693 & 58.2 & 42.7 & 41.6 & 3 & 0 \\
1.5 & 99.744 & 56.3 & 41.6 & 40.4 & 3 & 1 \\
1.6 & 99.788 & 55.8 & 40.6 & 40.1 & 2 & 1 \\
1.7 & 99.811 & 56.1 & 39.7 & 40.0 & 4 & 2 \\
1.8 & 99.829 & 55.8 & 39.7 & 39.7 & 4 & 1 \\
1.9 & 99.854 & 54.9 & 39.6 & 39.4 & 4 & 1 \\
2.0 & 99.865 & 54.6 & 39.0 & 39.1 & 5 & 2 \\
\bottomrule
\end{tabular*}
\endgroup
\end{table*}

\begin{table*}[t]
\centering
\caption{Minimum-search runtimes in microseconds per hex (lower is faster),
on the same unfiltered populations of 100,000 hexes per row.
These are \emph{different output problems}: Gmsh~4.15.2~\cite{gmshminsource} \texttt{minDetJac}
computes a lower bound on the \emph{whole-volume} minimum, while \textbf{Roots} and
\textbf{Roots+B} compute the minimum restricted to the \emph{boundary}.
The mathematical minima agree for strictly valid hexes but may differ
for invalid ones. Gmsh also refines maximum bounds and uses different
stopping rules. This comparator is Gmsh's minimum routine, not HXT's
boolean validity routine; HXT's depth-cutoff counts do not apply here.}
\label{tab:minimum-runtime}
\begingroup\small
\renewcommand{\arraystretch}{1.10}
\begin{tabular*}{\textwidth}{@{\extracolsep{\fill}}rrrr@{}}
\toprule
$\sigma$ & \shortstack{Gmsh volume lower bound\\($\mu$s/hex)} &
\shortstack{\textbf{Roots}: boundary minimum\\($\mu$s/hex)} &
\shortstack{\textbf{Roots+B}: boundary minimum\\($\mu$s/hex)} \\
\midrule
0.0 & 3.278 & 0.215 & 0.155 \\
0.1 & 3.758 & 1.321 & 0.224 \\
0.2 & 4.575 & 1.340 & 0.264 \\
0.3 & 5.500 & 1.362 & 0.313 \\
0.4 & 6.471 & 1.411 & 0.381 \\
0.5 & 7.277 & 1.425 & 0.440 \\
0.6 & 7.623 & 1.452 & 0.473 \\
0.7 & 7.905 & 1.471 & 0.493 \\
0.8 & 8.000 & 1.476 & 0.503 \\
0.9 & 8.086 & 1.478 & 0.505 \\
1.0 & 8.192 & 1.486 & 0.512 \\
1.1 & 8.162 & 1.492 & 0.512 \\
1.2 & 8.021 & 1.473 & 0.510 \\
1.3 & 8.107 & 1.470 & 0.507 \\
1.4 & 8.029 & 1.465 & 0.506 \\
1.5 & 8.012 & 1.461 & 0.502 \\
1.6 & 7.989 & 1.461 & 0.503 \\
1.7 & 7.942 & 1.463 & 0.500 \\
1.8 & 7.960 & 1.464 & 0.500 \\
1.9 & 8.037 & 1.468 & 0.502 \\
2.0 & 7.922 & 1.468 & 0.499 \\
\bottomrule
\end{tabular*}
\endgroup
\end{table*}

\subsection{Implementation and experimental setup}
We compare unchanged native HXT validity code with two variants of our
C boundary method. HXT is the released Gmsh C++ implementation at revision
\texttt{91b4154a}~\cite{hxtsource}, using the method
in~\cite{johnen2017}. Both of our variants use the
general-purpose C implementation of Algorithm~1010~\cite{orellana2020,quarticcode}
for quartic roots. Our first variant, \textbf{Roots}, runs without the Bernstein
face-bound pass in \autoref{sec:bernstein-face-bounds}. Our second variant, \textbf{Roots+B},
adds that optimization for speed.
Neither variant uses concavity screening. For this ablation, feasible stationary points of
negative curvature are retained too; the extra candidates cannot change
the minimum. The validity runs permit early exits after a negative
boundary witness. The minimum runs continue after negative witnesses.

For each $\sigma\in\{0,0.1,\ldots,2.0\}$, we generate 100,000 hexes
by adding independent Gaussian displacements of standard deviation
$\sigma$ to all twenty-four coordinates of the unit-cube vertices.
The same randomly drawn hexahedra are used across rows.
At $\sigma=0$, every element is the identity cube.
We use one thread on an Apple M4 Max, Apple Clang~16, C11/C++17,
\texttt{-O3}, and no floating-point reassociation or contraction.
Times are medians of nine trials with method order rotated between trials.

\subsection{Validity runtimes and HXT false negatives}
When recursive subdivision has not resolved validity, HXT's default guard,
\texttt{if (depth > 4) return false;}, stops the search and reports the hex
as invalid. In \autoref{tab:validity-throughput} the \emph{HXT depth cutoffs} column counts hexes rejected this way. The next column counts number of hexes for which this cutoff resulted in a false negative. This is verified doubly by our own method and by increasing the max depth in HXT.
In total, the guard rejects \SweepCutoffs{} of \SweepHexes{} inputs. Of these, \SweepFalseNegatives{} are valid.

HXT is faster on mildly perturbed, predominantly valid populations, though \textbf{Roots+B} is close behind.
\textbf{Roots+B} becomes faster at \(\sigma=0.5\), and \textbf{Roots} at
\(\sigma=0.8\). At high \(\sigma\), early rejection of invalid elements makes the two variants similarly fast. In general, all methods trend faster with more perturbed hexes due to earlier invalidity certificates.

\subsection{Minimum-search runtimes: different output problems}
Gmsh's \texttt{minDetJac} approximates the volume minimum by subdividing
the cube to tighten a Bernstein lower bound and a sampled upper
bound~\cite{gmshminsource}. Its stopping criterion requires the gap to be
below \(10^{-3}\) times the largest absolute sampled extremum, with
additional sign conditions. It returns the refined lower bound.
\autoref{tab:minimum-runtime} compares timing for their \texttt{minDetJac} routine with our measured boundary minimum.
Our boundary minimum and the volume minimum
agree for valid hexes by \autoref{thm:boundary-minimum}, but may differ
for invalid ones. 
Gmsh's default lower bound can be loose: among valid elements, it falls as much as
\(0.16123\) below our boundary minimum at \(\sigma=2.0\).
In a separate verification run, we tighten Gmsh's whole-volume bounding routine
to an absolute bracket tolerance of \(10^{-8}\).
Across all 400,179 valid elements, the largest gap between our boundary minimum
and Gmsh's tightented lower bound is \(9.975\times10^{-9}\), numerically confirming
boundary--volume minimum agreement to this accuracy.

\section{Conclusions}\label{sec:conclusions}
We prove Knupp's conjecture: strict boundary positivity guarantees strict
positivity throughout a trilinear hexahedron. Empirical observations have pointed to this conclusion for a long time but now we have a theoretical guarantee. Our followup results lead to a newly enabled practical algorithm for determining hexahedral validity that is competetive with current methods in both runtime and accuracy.

\setcounter{secnumdepth}{3}
\paragraph{Limitations.}
The method computes the boundary minimum, which need not equal the full
volume minimum for an invalid hexahedron; that value is unnecessary for
validity checking. The proofs assume exact arithmetic, while the
implementation uses double precision.
Agreement on this synthetic perturbation sweep
does not certify all floating-point inputs.

\clearpage

\bibliographystyle{eg-alpha-doi}
\bibliography{references}

\par\bigskip
\Needspace{12\baselineskip}
\appendix
\setcounter{secnumdepth}{3}
\section{Deferred proofs}\label{app:lemma-proofs}

\subsection{Proof of \autoref{lemma-1}}\label{proof-1}

\renewcommand{\restatedlemmaname}{\autoref{lemma-1} (restated)}
\begin{restatedlemma}[{Each connected component of \(Z\) is a compact smooth surface enclosing a bounded region}]

Under the boundary-positivity assumption, \(p\in Z \rightarrow \nabla J(p)\ne0\). Consequently, each connected component \(S\) of \(Z\) is a compact, smooth, two-dimensional embedded surface without boundary. The surface \(S\) is the boundary of a bounded three-dimensional region \(\Omega\subset\mathbb R^3\).

\end{restatedlemma}

\setcounter{lemma}{1}
\setcounter{sublemma}{0}

\begin{proof}

We first prove that a zero of \(J\) cannot also be a zero of \(\nabla J\). Suppose, for a contradiction, that a point \(p\in Q\) satisfies
\begin{equation}
J(p)=0,\qquad \nabla J(p)=0.\label{web-equation-06}
\end{equation}

Since \(J>0\) on \(\partial Q\), the point \(p\) lies strictly inside the cube.

Our information about positive values of \(J\) is on the cube boundary. To use that information at \(p\), consider the three coordinate squares through \(p\): each is obtained by fixing one reference coordinate at its value at \(p\) and allowing the other two coordinates to range over \([0,1]\). The four edges of each square lie on \(\partial Q\), so \(J\) is positive on each square's perimeter.

We want those positive perimeter values to contradict \(J(p)=0\). The following fact identifies a case in which positivity on the perimeter does force positivity throughout the square: the image of the square under \(F\) lies in a plane.

\begin{sublemma}\label{lemma-1a}

Let \(R\subset Q\) be a coordinate square obtained by fixing one reference coordinate. If \(F(R)\) is contained in an affine plane in physical space and \(J>0\) on the perimeter of \(R\), then \(J>0\) everywhere on \(R\).

\end{sublemma}

\noindent\emph{The proof of \autoref{lemma-1a} is in \autoref{proof-1a}.}

It follows that \textbf{none of the three coordinate squares through \(p\) can have a planar image under \(F\)}. If one did, \autoref{lemma-1a} would give \(J(p)>0\), contradicting our assumption \(J(p)=0\).

We have not yet used \(\nabla J(p)=0\). The next fact shows that this condition, together with \(J(p)=0\), forces exactly the planar slice we have just ruled out.

\begin{sublemma}[{\(J(p)=0\) and \(\nabla J(p)=0\) force a planar image for a coordinate slice through \(p\)}]\label{lemma-1b}

For any trilinear map \(F\), if \(p\in(0,1)^3\) satisfies \(J(p)=0\) and \(\nabla J(p)=0\), then at least one of the three coordinate squares through \(p\) has an image under \(F\) contained in an affine plane.

\end{sublemma}

\noindent\emph{The proof of \autoref{lemma-1b} is in \autoref{proof-1b}.}

\autoref{lemma-1b} produces a planar coordinate square through our hypothesized point \(p\). Its perimeter lies on \(\partial Q\), where \(J>0\). \autoref{lemma-1a} therefore gives \(J(p)>0\), contradicting \(J(p)=0\).

Thus no zero of \(J\) has zero gradient:
\begin{equation}
\boxed{\nabla J(q)\ne0\qquad\forall q\in Z.}\label{web-equation-30}
\end{equation}

Now we apply the \textbf{regular level set theorem} in the following scalar-valued form: let \(O\subset\mathbb R^3\) be open and let \(f:O\to\mathbb R\) be smooth. If \(\nabla f(q)\ne0\) at every \(q\in f^{-1}(\{0\})\), then \(f^{-1}(\{0\})\) is a smooth, two-dimensional embedded submanifold without boundary. See Corollary~5.14~\cite[p.~106]{lee2013}.

To apply this theorem, take \(O=(0,1)^3\) and \(f=J\) on \(O\). Boundary positivity ensures that \(Z\) has no points on \(\partial Q\), so
\begin{equation}
Z=\{q\in O:J(q)=0\}.\label{web-equation-31}
\end{equation}

The function \(J\) is smooth because it is a polynomial, and we have just proved \(\nabla J\ne0\) at every point of \(Z\). The theorem therefore makes \(Z\) a smooth, two-dimensional embedded submanifold without boundary, possibly disconnected. Each connected component \(S\) of \(Z\) is consequently a smooth embedded surface without boundary.

Since \(Z\) is closed in compact \(Q\), each connected component \(S\) is compact, so \(x\) attains its minimum and maximum on \(S\). As a compact, connected, smooth embedded surface without boundary in \(\mathbb R^3\), \(S\) bounds a bounded region \(\Omega\), whose outward unit normal we use.

\end{proof}

\subsection{Proof of \autoref{lemma-1a}}\label{proof-1a}

See \autoref{lemma-1a} for the statement.

\begin{proof}

Fix \(t\in[0,1]\) and let \(R=\{t\}\times[0,1]^2\). Assume that \(F(R)\) lies in an affine plane and that \(J>0\) on the perimeter of \(R\). We will prove that \(J>0\) throughout \(R\); the other two slice directions follow by cyclically permuting the coordinates. We first factor \(J\) on this slice. Choose a fixed unit normal \(n\) to the plane containing \(F(R)\). The vectors \(\frac{\partial F}{\partial y}(t,y,z)\) and \(\frac{\partial F}{\partial z}(t,y,z)\) are parallel to that plane, so their cross product is a scalar multiple of \(n\):
\begin{equation}
\begin{aligned}
J(t,y,z)
&=\det\left[\frac{\partial F}{\partial x}(t,y,z)\;\middle|\;
\frac{\partial F}{\partial y}(t,y,z)\;\middle|\;
\frac{\partial F}{\partial z}(t,y,z)\right]\\[6pt]
&=\frac{\partial F}{\partial x}(t,y,z)\cdot
\underbrace{\left(\frac{\partial F}{\partial y}(t,y,z)
\times\frac{\partial F}{\partial z}(t,y,z)\right)}_{=:N(y,z):=\ell(y,z)n}\\[6pt]
&=\ell(y,z)\,\underbrace{\left(\frac{\partial F}{\partial x}(t,y,z)\cdot n\right)}_{=:q(y,z)}\\[6pt]
&=\ell(y,z)\,q(y,z).
\end{aligned}\label{planar-slice-factorization}
\end{equation}

We will now show that \(\ell\) is affine and \(q\) is bilinear in \((y,z)\). These properties will let us prove positivity throughout the slice from positivity at its four corners.

To show \(\ell\) is affine, expand \(F\) on the fixed slice:
\begin{equation}
H(y,z):=F(t,y,z)=p_0+ay+bz+cyz.\label{web-equation-08}
\end{equation}

The coefficient vectors \(p_0,a,b,c\in\mathbb R^3\) are constant on this slice. The vectors \(a,b,c\) are parallel to the image plane, because each is a linear combination of differences of the four corner images.

Then we expand \(N\) as follows:
\begin{equation}
\begin{aligned}
N(y,z)&:=\frac{\partial F}{\partial y}(t,y,z)\times\frac{\partial F}{\partial z}(t,y,z)\\
&=\underbrace{\frac{\partial H}{\partial y}(y,z)}_{a+cz}
\times\underbrace{\frac{\partial H}{\partial z}(y,z)}_{b+cy}\\
&=(a+cz)\times(b+cy)\\
&=a\times b+y(a\times c)+z(c\times b)
+yz\underbrace{\cancel{(c\times c)}}_{0}\\
&=a\times b+y(a\times c)+z(c\times b).
\end{aligned}\label{web-equation-09}
\end{equation}

Since \(N(y,z)=\ell(y,z)n\) and \(n\) is a fixed unit vector,
\begin{equation}
\begin{aligned}
\ell(y,z)&=N(y,z)\cdot n\\
&=(a\times b)\cdot n+y(a\times c)\cdot n+z(c\times b)\cdot n.
\end{aligned}\label{affine-slice-factor}
\end{equation}

The coefficients of \(1,y,z\) in \autoref{affine-slice-factor} are constant, so \(\ell\) is affine.

The function \(q\) is bilinear in \((y,z)\), because \(F\) is trilinear and \(n\) is constant.

We must now show that both factors are positive. Since \(J=\ell q>0\) on the perimeter, \(\ell\) never vanishes there. The perimeter is connected and \(\ell\) is continuous, so \(\ell\) has one sign on the entire perimeter. Note, we can reverse the sign choice of \(n\) as needed to make \(\ell>0\) on the perimeter. Flipping the sign of \(n\) also flips the sign of \(q\) though so this does not change \(J=\ell q>0\) on the perimeter.

Since \(\ell\) is affine and positive at the four corners, \(\ell>0\) throughout the square. At the four corners, \(q=J/\ell>0\); since \(q\) is bilinear in the square's coordinates \((y,z)\), \(q>0\) throughout the square. Consequently \(J=\ell q>0\) throughout the square.

Cyclically permuting the reference coordinates proves the other two cases. A cyclic permutation preserves the determinant's sign, so the same factorization argument applies to a fixed-\(y\) or fixed-\(z\) slice.
\end{proof}

\subsection{Proof of \autoref{lemma-1b}}\label{proof-1b}

See \autoref{lemma-1b} for the statement.

\begin{proof}

Let \(p=(x_0,y_0,z_0)\in(0,1)^3\) satisfy \(J(p)=0\) and \(\nabla J(p)=0\). We will show that at least one of the three coordinate squares through \(p\) has a planar image under \(F\). To do so, we first express each slice using three vectors-valued coefficients; linear dependence of those vectors will place the entire slice image in a plane.

Define the following derivative vectors, all evaluated at the fixed point \(p\):
\begin{equation}
U:=\frac{\partial F}{\partial x}(p),\qquad V:=\frac{\partial F}{\partial y}(p),\qquad W:=\frac{\partial F}{\partial z}(p),\label{web-equation-11}
\end{equation}
\begin{equation}
A:=\frac{\partial^2F}{\partial x\,\partial y}(p),\qquad B:=\frac{\partial^2F}{\partial x\,\partial z}(p),\qquad C:=\frac{\partial^2F}{\partial y\,\partial z}(p).\label{web-equation-12}
\end{equation}

Let \(s,t\) denote displacements in the two coordinates allowed to vary on a slice. Since the restrictions of \(F\) to the slices are bilinear, their exact formulas are
\begin{equation}
F(x_0,y_0+s,z_0+t)=F(p)+Vs+Wt+Cst,\label{web-equation-13}
\end{equation}
\begin{equation}
F(x_0+s,y_0,z_0+t)=F(p)+Us+Wt+Bst,\label{web-equation-14}
\end{equation}
\begin{equation}
F(x_0+s,y_0+t,z_0)=F(p)+Us+Vt+Ast.\label{web-equation-15}
\end{equation}

These formulas hold for all \(s,t\) that keep the reference point in \(Q\). Each image lies in \(F(p)\) plus the span of its three coefficient vectors. Thus it is enough to show that at least one of the triples
\begin{equation}
(V,W,C),\qquad (U,W,B),\qquad (U,V,A)\label{web-equation-16}
\end{equation}

is linearly dependent.

\textbf{The condition \(J(p)=0\) makes \(U,V,W\) linearly dependent.} By definition,
\begin{equation}
0=J(p)=\det([U\mid V\mid W]).\label{web-equation-17}
\end{equation}

If \(V,W\) are dependent then the first slice formula already has planar image, and the proof is finished.

It remains to treat the case where \(V,W\) are independent. The zero determinant then forces \(U\) into the plane spanned by \(V,W\). Consequently there are unique scalars \(\alpha,\beta\) with
\begin{equation}
U=\alpha V+\beta W.\label{web-equation-18}
\end{equation}

These coefficients immediately identify two more cases in which a slice is planar. If \(\alpha=0\), then \(U,W\) are dependent, so the second slice formula has a planar image. If \(\beta=0\), then \(U,V\) are dependent, so the third slice formula has a planar image.

We are therefore left only with the case
\begin{equation}
V,W\text{ independent},\qquad \alpha\ne0,\qquad \beta\ne0.\label{web-equation-19}
\end{equation}

In this remaining case, we will prove that the first slice is planar by showing that \(C\) also lies in the plane spanned by \(V,W\). Since \(V,W\) are independent, this is equivalent to proving
\begin{equation}
\kappa:=\det([V\mid W\mid C])=0.\label{web-equation-20}
\end{equation}

We have used \(J(p)=0\). The remaining assumption is \(\nabla J(p)=0\). We first expand its \(y\)- and \(z\)-component equations and substitute \(U=\alpha V+\beta W\). We do this in two columns in parallel on the next page formatted specifically to demonstrate the parallels between these two steps.

\clearpage
\Needspace{43\baselineskip}
\begin{samepage}
\subsubsection*{Starting with \(\frac{\partial J}{\partial y}(p)=0\).}

We apply the determinant product rule \cite[Eq.~(2.3)]{grover2010} to expand \(\frac{\partial J}{\partial y}(p)\). Since \(F\) is trilinear, the pure second derivative \(\frac{\partial^2F}{\partial y^2}(p)=0\):
\begin{equation}
\begin{aligned}
&0=\frac{\partial J}{\partial y}(p)\\[4pt]
&\quad=\det\left[\underbrace{\frac{\partial^2F}{\partial x\,\partial y}(p)}_{A}\;\middle|\;\underbrace{\frac{\partial F}{\partial y}(p)}_{V}\;\middle|\;\underbrace{\frac{\partial F}{\partial z}(p)}_{W}\right]\\[6pt]
&\qquad+\det\left[\underbrace{\frac{\partial F}{\partial x}(p)}_{U}\;\middle|\;\underbrace{\frac{\partial^2F}{\partial y^2}(p)}_{0}\;\middle|\;\underbrace{\frac{\partial F}{\partial z}(p)}_{W}\right]\\[6pt]
&\qquad+\det\left[\underbrace{\frac{\partial F}{\partial x}(p)}_{U}\;\middle|\;\underbrace{\frac{\partial F}{\partial y}(p)}_{V}\;\middle|\;\underbrace{\frac{\partial^2F}{\partial y\,\partial z}(p)}_{C}\right]\\[6pt]
&\quad=\det([A\mid V\mid W])+\underbrace{\cancel{\det([U\mid 0\mid W])}}_{0}\\
&\qquad+\det([U\mid V\mid C])\\
&\quad=\det([A\mid V\mid W])+\det([U\mid V\mid C]).
\end{aligned}\label{web-equation-21}
\end{equation}

Now substitute \(U=\alpha V+\beta W\) into the second determinant:
\begin{equation}
\begin{aligned}
&\det([U\mid V\mid C])\\
&\quad=\det([\underbrace{\alpha V+\beta W}_{U}\mid V\mid C])\\
&\quad=\alpha\underbrace{\cancel{\det([V\mid V\mid C])}}_{0}
+\beta\det([W\mid V\mid C])\\
&\quad=\beta\det([W\mid V\mid C])\\
&\quad=-\beta\underbrace{\det([V\mid W\mid C])}_{\kappa}\\
&\quad=-\beta\kappa.
\end{aligned}\label{beta-column-determinant}
\end{equation}

Consequently,
\begin{equation}
\begin{aligned}
\frac{\partial J}{\partial y}(p)&=0=\det([A\mid V\mid W])-\beta\kappa,\\
\det([A\mid V\mid W])&=\beta\kappa.
\end{aligned}\label{y-derivative-determinant-value}
\end{equation}
\end{samepage}

\Needspace{43\baselineskip}
\begin{samepage}
\subsubsection*{Starting with \(\frac{\partial J}{\partial z}(p)=0\).}

We apply the determinant product rule \cite[Eq.~(2.3)]{grover2010} to expand \(\frac{\partial J}{\partial z}(p)\). Since \(F\) is trilinear, the pure second derivative \(\frac{\partial^2F}{\partial z^2}(p)=0\):
\begin{equation}
\begin{aligned}
&0=\frac{\partial J}{\partial z}(p)\\[4pt]
&\quad=\det\left[\underbrace{\frac{\partial^2F}{\partial x\,\partial z}(p)}_{B}\;\middle|\;\underbrace{\frac{\partial F}{\partial y}(p)}_{V}\;\middle|\;\underbrace{\frac{\partial F}{\partial z}(p)}_{W}\right]\\[6pt]
&\qquad+\det\left[\underbrace{\frac{\partial F}{\partial x}(p)}_{U}\;\middle|\;\underbrace{\frac{\partial^2F}{\partial y\,\partial z}(p)}_{C}\;\middle|\;\underbrace{\frac{\partial F}{\partial z}(p)}_{W}\right]\\[6pt]
&\qquad+\det\left[\underbrace{\frac{\partial F}{\partial x}(p)}_{U}\;\middle|\;\underbrace{\frac{\partial F}{\partial y}(p)}_{V}\;\middle|\;\underbrace{\frac{\partial^2F}{\partial z^2}(p)}_{0}\right]\\[6pt]
&\quad=\det([B\mid V\mid W])+\det([U\mid C\mid W])\\
&\qquad+\underbrace{\cancel{\det([U\mid V\mid 0])}}_{0}\\
&\quad=\det([B\mid V\mid W])+\det([U\mid C\mid W]).
\end{aligned}\label{web-equation-24}
\end{equation}

Now substitute \(U=\alpha V+\beta W\) into the second determinant:
\begin{equation}
\begin{aligned}
&\det([U\mid C\mid W])\\
&\quad=\det([\underbrace{\alpha V+\beta W}_{U}\mid C\mid W])\\
&\quad=\alpha\det([V\mid C\mid W])
+\beta\underbrace{\cancel{\det([W\mid C\mid W])}}_{0}\\
&\quad=\alpha\det([V\mid C\mid W])\\
&\quad=-\alpha\underbrace{\det([V\mid W\mid C])}_{\kappa}\\
&\quad=-\alpha\kappa.
\end{aligned}\label{alpha-column-determinant}
\end{equation}

Consequently,
\begin{equation}
\begin{aligned}
\frac{\partial J}{\partial z}(p)&=0=\det([B\mid V\mid W])-\alpha\kappa,\\
\det([B\mid V\mid W])&=\alpha\kappa.
\end{aligned}\label{z-derivative-determinant-value}
\end{equation}
\end{samepage}

\Needspace{32\baselineskip}
We now calculate \(\frac{\partial J}{\partial x}(p)\) in terms of \(\alpha,\beta,\kappa\). We will use \(\frac{\partial J}{\partial x}(p)=0\) after obtaining that formula.
\begin{equation}
\begin{aligned}
&\frac{\partial J}{\partial x}(p)\\[4pt]
&\quad=\det\left[\underbrace{\frac{\partial^2F}{\partial x^2}(p)}_{0}\;\middle|\;\underbrace{\frac{\partial F}{\partial y}(p)}_{V}\;\middle|\;\underbrace{\frac{\partial F}{\partial z}(p)}_{W}\right]\\[6pt]
&\qquad+\det\left[\underbrace{\frac{\partial F}{\partial x}(p)}_{U}\;\middle|\;\underbrace{\frac{\partial^2F}{\partial x\,\partial y}(p)}_{A}\;\middle|\;\underbrace{\frac{\partial F}{\partial z}(p)}_{W}\right]\\[6pt]
&\qquad+\det\left[\underbrace{\frac{\partial F}{\partial x}(p)}_{U}\;\middle|\;\underbrace{\frac{\partial F}{\partial y}(p)}_{V}\;\middle|\;\underbrace{\frac{\partial^2F}{\partial x\,\partial z}(p)}_{B}\right]\\[6pt]
&\quad=\underbrace{\cancel{\det([0\mid V\mid W])}}_{0}+\det([U\mid A\mid W])\\
&\qquad+\det([U\mid V\mid B])\\
&\quad=\det([U\mid A\mid W])+\det([U\mid V\mid B]).
\end{aligned}\label{x-derivative-expansion}
\end{equation}

Substitute \(U=\alpha V+\beta W\) in each determinant:
\begin{equation}
\begin{aligned}
&\det([U\mid A\mid W])\\
&=\det([\underbrace{\alpha V+\beta W}_{U}\mid A\mid W])\\
&=\alpha\det([V\mid A\mid W])
+\beta\underbrace{\cancel{\det([W\mid A\mid W])}}_{0}\\
&=-\alpha\det([A\mid V\mid W]),\\[8pt]
&\det([U\mid V\mid B])\\
&=\det([\underbrace{\alpha V+\beta W}_{U}\mid V\mid B])\\
&=\alpha\underbrace{\cancel{\det([V\mid V\mid B])}}_{0}
+\beta\det([W\mid V\mid B])\\
&=-\beta\det([B\mid V\mid W]).
\end{aligned}\label{x-derivative-column-substitutions}
\end{equation}

Using \autoref{y-derivative-determinant-value} and \autoref{z-derivative-determinant-value},
\begin{equation}
\begin{aligned}
\frac{\partial J}{\partial x}(p)
&=-\alpha\underbrace{\det([A\mid V\mid W])}_{\beta\kappa}
-\beta\underbrace{\det([B\mid V\mid W])}_{\alpha\kappa}\\
&=-\alpha(\beta\kappa)-\beta(\alpha\kappa)\\
&=-2\alpha\beta\kappa.
\end{aligned}\label{x-derivative-factorization}
\end{equation}

The derivation of \autoref{x-derivative-factorization} requires only trilinearity, the relation \(U=\alpha V+\beta W\), and \(\frac{\partial J}{\partial y}(p)=\frac{\partial J}{\partial z}(p)=0\). It therefore applies at any point satisfying those conditions. We have not used \(\frac{\partial J}{\partial x}(p)=0\) in deriving the identity, and no division by \(\alpha\) or \(\beta\) occurred, so the identity also holds when either coefficient is zero.

For \autoref{lemma-1b}, we now use the additional assumption \(\frac{\partial J}{\partial x}(p)=0\). \autoref{x-derivative-factorization} gives \(0=-2\alpha\beta\kappa\). We are in the case where \(\alpha\ne0\) and \(\beta\ne0\), so \(\kappa=0\). Hence \(V,W,C\) are linearly dependent. The first slice formula places the image of the fixed-\(x\) square in an affine plane.

Every case has produced a planar coordinate slice through \(p\). This proves \autoref{lemma-1b}.
\end{proof}

\subsection{Proof of \autoref{lemma-2}}\label{proof-2}

\renewcommand{\restatedlemmaname}{\autoref{lemma-2} (restated)}
\begin{restatedlemma}[{On \(S\), \(\frac{\partial J}{\partial x}\) has one nonzero sign wherever \(\frac{\partial J}{\partial y}=\frac{\partial J}{\partial z}=0\)}]

Under the boundary-positivity assumption, let \(S\) be a connected component of \(Z\). Suppose \(p,q\in S\) satisfy
\begin{equation}
\begin{aligned}
\frac{\partial J}{\partial y}(p)&=\frac{\partial J}{\partial z}(p)=0,\\
\frac{\partial J}{\partial y}(q)&=\frac{\partial J}{\partial z}(q)=0.
\end{aligned}\label{restated-web-equation-37}
\end{equation}

\Needspace{4\baselineskip}
Then
\begin{equation}
\boxed{\frac{\partial J}{\partial x}(p)\frac{\partial J}{\partial x}(q)>0.}\label{restated-web-equation-38}
\end{equation}

\end{restatedlemma}

\setcounter{lemma}{2}
\setcounter{sublemma}{0}

\begin{proof}

We reuse the notation introduced earlier:
\begin{equation}
U=\frac{\partial F}{\partial x},\qquad
V=\frac{\partial F}{\partial y},\qquad
W=\frac{\partial F}{\partial z},\qquad
C=\frac{\partial^2F}{\partial y\,\partial z}.
\label{surface-derivative-definitions}
\end{equation}

To compare the signs of \(\frac{\partial J}{\partial x}\) at \(p,q\), we will reuse \autoref{x-derivative-factorization} from the proof of \autoref{lemma-1b} in \autoref{proof-1b}. That identity requires us to establish the relation \(U(r)=\alpha(r)V(r)+\beta(r)W(r)\) first.

Since \(J(r)=0\) on \(S\), independence of \(V(r),W(r)\) will force \(U(r)\) into their span. The following statement establishes this independence at every point of \(S\).

\begin{sublemma}[{Any two columns of \(DF(r)\) are independent for every \(r\in Z\)}]\label{lemma-2a}

Under the boundary-positivity assumption, every pair of columns of \(DF(r)\) is linearly independent at every \(r\in Z\).

\end{sublemma}

\noindent\emph{The proof of \autoref{lemma-2a} is in \autoref{proof-2a}.}

Since \(J(r)=0\) and no pair among \(U(r),V(r),W(r)\) is dependent (\autoref{lemma-2a}), there are unique nonzero scalars \(\alpha(r),\beta(r)\) such that
\begin{equation}
U(r)=\alpha(r)V(r)+\beta(r)W(r)\qquad(r\in S).\label{column-dependence}
\end{equation}

Now fix any \(r\in S\) with \(\frac{\partial J}{\partial y}(r)=\frac{\partial J}{\partial z}(r)=0\). Every condition used to derive \autoref{x-derivative-factorization} is satisfied:

\begin{itemize}

\item \(F\) is trilinear, by assumption.

\item \(U(r)=\alpha(r)V(r)+\beta(r)W(r)\), by \autoref{column-dependence}.

\item \(\frac{\partial J}{\partial y}(r)=\frac{\partial J}{\partial z}(r)=0\), by the choice of \(r\).

\end{itemize}

With \(\kappa(r)=\det([V(r)\mid W(r)\mid C(r)])\), \autoref{x-derivative-factorization} therefore gives
\begin{equation}
\frac{\partial J}{\partial x}(r)=-2\alpha(r)\beta(r)\kappa(r).\label{surface-derivative-factorization}
\end{equation}

This application does not assume \(\frac{\partial J}{\partial x}(r)=0\). The identity holds at every \(r\in S\) with the two stated zero partial derivatives, including the points \(p,q\) in \autoref{lemma-2}.

To complete the sign comparison, we need \(\alpha,\beta,\kappa\) to be continuous and nonzero everywhere on \(S\). \autoref{lemma-2a} and \autoref{column-dependence} already ensure that \(\alpha,\beta\) never vanish. We next prove that \(\kappa\) never vanishes.

\begin{sublemma}[{\(\kappa(r)\ne0\) for every \(r\in S\)}]\label{lemma-2b}

Under the boundary-positivity assumption,
\begin{equation}
\kappa(r)=\det([V(r)\mid W(r)\mid C(r)])\ne0\qquad\forall r\in S.\label{web-equation-42}
\end{equation}

\end{sublemma}

\noindent\emph{The proof of \autoref{lemma-2b} is in \autoref{proof-2b}.}

To prove continuity of \(\alpha\) and \(\beta\), we express each function as a quotient with the nonzero denominator \(\kappa\).

We start from \autoref{alpha-column-determinant} and \autoref{beta-column-determinant}, whose only assumptions are the definition \(\kappa(r)=\det([V(r)\mid W(r)\mid C(r)])\) and the relation \(U(r)=\alpha(r)V(r)+\beta(r)W(r)\). Both conditions have already been established, with the relation for \(U\) given in \autoref{column-dependence}:
\begin{equation}
\begin{aligned}
\det([U(r)\mid W(r)\mid C(r)])
&=-\det([U(r)\mid C(r)\mid W(r)])\\
&=\alpha(r)\kappa(r).
\end{aligned}\label{alpha-determinant-expansion}
\end{equation}
\begin{equation}
\begin{aligned}
\det([V(r)\mid U(r)\mid C(r)])
&=-\det([U(r)\mid V(r)\mid C(r)])\\
&=\beta(r)\kappa(r).
\end{aligned}\label{beta-determinant-expansion}
\end{equation}

Since \(\kappa(r)\ne0\) on \(S\), division gives the following formulas for \(\alpha(r)\) and \(\beta(r)\):
\begin{equation}
\begin{aligned}
\alpha(r)&=\frac{\det([U(r)\mid W(r)\mid C(r)])}{\kappa(r)},\\[4pt]
\beta(r)&=\frac{\det([V(r)\mid U(r)\mid C(r)])}{\kappa(r)}.
\end{aligned}\label{coefficient-quotients}
\end{equation}

Every vector \(U,V,W,C\) is a polynomial function of \(r\). Hence \(\kappa\) is continuous, and \autoref{coefficient-quotients} makes \(\alpha,\beta\) continuous on \(S\). We also have \(\alpha(r)\ne0\) and \(\beta(r)\ne0\) by \autoref{lemma-2a} and \autoref{column-dependence}, and \(\kappa(r)\ne0\) by \autoref{lemma-2b}, for every \(r\in S\). Since \(S\) is connected, each of \(\alpha,\beta,\kappa\) has one fixed nonzero sign on \(S\).

Finally, \autoref{surface-derivative-factorization} applies at both \(p\) and \(q\):
\begin{equation}
\begin{aligned}
\frac{\partial J}{\partial x}(p)&=-2\alpha(p)\beta(p)\kappa(p),\\
\frac{\partial J}{\partial x}(q)&=-2\alpha(q)\beta(q)\kappa(q).
\end{aligned}\label{web-equation-47}
\end{equation}

Each factor has the same nonzero sign at \(p\) and \(q\). Thus the two derivatives have the same nonzero sign, proving
\begin{equation}
\frac{\partial J}{\partial x}(p)\frac{\partial J}{\partial x}(q)>0.\qedhere\label{web-equation-48}
\end{equation}
\end{proof}

\subsection{Proof of \autoref{lemma-2a}}\label{proof-2a}

See \autoref{lemma-2a} for the statement.

\begin{proof}

Fix \(r=(x_0,y_0,z_0)\in Z\). To see what dependence of \(V(r),W(r)\) would imply for the fixed-\(x\) slice, write its exact bilinear formula:
\begin{equation}
\begin{aligned}
F(x_0,y_0+s,z_0+t)
={}&F(r)+V(r)s+W(r)t\\
&+\underbrace{\frac{\partial^2F}{\partial y\,\partial z}(r)}_{=:C(r)}\,st.
\end{aligned}\label{fixed-x-slice-expansion}
\end{equation}

The formula holds for all displacements \(s,t\) that keep the reference point in \(Q\). If \(V(r),W(r)\) were dependent, their span would have dimension at most one. Adding the single vector \(C(r)\) would raise that dimension by at most one. The entire slice image would therefore lie in an affine plane through \(F(r)\).

By \autoref{lemma-1a}, a coordinate square with a planar image and positive \(J\) on its perimeter has positive \(J\) throughout the square. Here the perimeter lies on \(\partial Q\), where \(J>0\), but the square contains \(r\), where \(J(r)=0\). This contradiction proves that \(V(r),W(r)\) are independent.

Cyclically permuting the reference coordinates gives the same conclusion for any pair among \(U,V,W\).
\end{proof}

\subsection{Proof of \autoref{lemma-2b}}\label{proof-2b}

See \autoref{lemma-2b} for the statement.

\begin{proof}

Fix \(r=(x_0,y_0,z_0)\in S\). We repeat the fixed-\(x\) slice formula \autoref{fixed-x-slice-expansion} here for completeness:
\begin{equation}
\begin{aligned}
F(x_0,y_0+s,z_0+t)
={}&F(r)+V(r)s+W(r)t\\
&+\underbrace{\frac{\partial^2F}{\partial y\,\partial z}(r)}_{=C(r)}\,st.
\end{aligned}\label{web-equation-43}
\end{equation}

If \(\kappa(r)=0\), the three coefficient vectors \(V(r),W(r),C(r)\) would span at most two dimensions, so the whole slice image would lie in an affine plane through \(F(r)\). The perimeter lies on \(\partial Q\), where \(J>0\). By \autoref{lemma-1a}, positivity would extend throughout the slice, contradicting \(J(r)=0\). Hence \(\kappa(r)\ne0\).
\end{proof}

\subsection{Proof of \autoref{thm:boundary-minimum}}\label{proof:boundary-minimum}

\renewcommand{\restatedlemmaname}{\autoref{thm:boundary-minimum} (restated)}
\begin{restatedlemma}[Boundary minimum]
If \(J>0\) on \(\partial Q\), then
\begin{equation}
\min_{r\in Q}J(r)=\min_{r\in\partial Q}J(r)>0.
\label{eq:boundary-minimum-restated}
\end{equation}
\end{restatedlemma}

\begin{proof}
By \autoref{thm:knupp-conjecture}, \(J>0\) throughout \(Q\). Since \(J\) is continuous and \(Q\) is compact, \(J\) attains a positive global minimum. If a global minimizer lies on \(\partial Q\), the conclusion already holds. To handle an interior global minimizer, we will show that every positive interior local minimum extends unchanged along a coordinate-line segment to the boundary.

Let \(p\in(0,1)^3\) be a local minimizer with \(J(p)>0\). We will expand \(J\) along the three coordinate lines through \(p\). Each expansion is an exact polynomial of degree at most two. To make one of these polynomials constant, we must show that its linear and quadratic coefficients both vanish.

At an interior local minimum, the gradient vanishes and the Hessian (the matrix of second partial derivatives) is positive semidefinite:
\begin{equation}
\nabla J(p)=0,\qquad \nabla^2J(p)\succeq0.\label{local-minimum-conditions}
\end{equation}

The gradient condition will make all three linear coefficients zero. The Hessian condition will make all three quadratic coefficients nonnegative. We will then use the determinant formulas for those coefficients to prove that at least one must be zero.

To compute the coordinate-line polynomials, evaluate the first and mixed second derivatives at the fixed point \(p\):
\begin{equation}
\begin{aligned}
U&=\frac{\partial F}{\partial x}(p),&
V&=\frac{\partial F}{\partial y}(p),&
W&=\frac{\partial F}{\partial z}(p),\\[4pt]
A&=\frac{\partial^2F}{\partial x\,\partial y}(p),&
B&=\frac{\partial^2F}{\partial x\,\partial z}(p),&
C&=\frac{\partial^2F}{\partial y\,\partial z}(p).
\end{aligned}
\label{minimum-derivative-definitions}
\end{equation}
For a displacement \(h=(h_1,h_2,h_3)\), trilinearity gives the exact expansion
\begin{equation}
\begin{aligned}
F(p+h)&=F(p)+Uh_1+Vh_2+Wh_3\\
&\quad+Ah_1h_2+Bh_1h_3+Ch_2h_3\\
&\quad+\frac{\partial^3F}{\partial x\,\partial y\,\partial z}(p)\,h_1h_2h_3.
\end{aligned}\label{trilinear-expansion-at-minimum}
\end{equation}

Let \(e_1,e_2,e_3\) be the standard coordinate vectors. Differentiating \autoref{trilinear-expansion-at-minimum} with respect to all three displacement coordinates and then evaluating on each coordinate line gives
\begin{equation}
\begin{aligned}
DF(p+te_1)&=[U\mid V+tA\mid W+tB],\\[4pt]
DF(p+te_2)&=[U+tA\mid V\mid W+tC],\\[4pt]
DF(p+te_3)&=[U+tB\mid V+tC\mid W].
\end{aligned}\label{coordinate-line-derivative-matrices}
\end{equation}

Expanding each determinant by linearity in its columns gives the constant, linear, and quadratic coefficients explicitly:
\begin{equation}
\begin{aligned}
J(p+te_1)&=\det([U\mid V+tA\mid W+tB])\\
&=\underbrace{\det([U\mid V\mid W])}_{J(p)}\\
&\quad+t\bigl(\det([U\mid A\mid W])+\det([U\mid V\mid B])\bigr)\\
&\quad+t^2\underbrace{\det([U\mid A\mid B])}_{=:q_1}.
\end{aligned}\label{x-axis-jacobian-polynomial}
\end{equation}
\begin{equation}
\begin{aligned}
J(p+te_2)&=\det([U+tA\mid V\mid W+tC])\\
&=\underbrace{\det([U\mid V\mid W])}_{J(p)}\\
&\quad+t\bigl(\det([A\mid V\mid W])+\det([U\mid V\mid C])\bigr)\\
&\quad+t^2\underbrace{\det([A\mid V\mid C])}_{=:q_2}.
\end{aligned}\label{y-axis-jacobian-polynomial}
\end{equation}
\begin{equation}
\begin{aligned}
J(p+te_3)&=\det([U+tB\mid V+tC\mid W])\\
&=\underbrace{\det([U\mid V\mid W])}_{J(p)}\\
&\quad+t\bigl(\det([B\mid V\mid W])+\det([U\mid C\mid W])\bigr)\\
&\quad+t^2\underbrace{\det([B\mid C\mid W])}_{=:q_3}.
\end{aligned}\label{z-axis-jacobian-polynomial}
\end{equation}

Differentiating these polynomials at \(t=0\) picks out their linear coefficients. The condition \(\nabla J(p)=0\) therefore gives
\begin{equation}
\begin{aligned}
0=\frac{\partial J}{\partial x}(p)
&=\det([U\mid A\mid W])+\det([U\mid V\mid B]),\\[6pt]
0=\frac{\partial J}{\partial y}(p)
&=\det([A\mid V\mid W])+\det([U\mid V\mid C]),\\[6pt]
0=\frac{\partial J}{\partial z}(p)
&=\det([B\mid V\mid W])+\det([U\mid C\mid W]).
\end{aligned}\label{axis-linear-coefficients-zero}
\end{equation}

Differentiating twice gives the diagonal entries of the Hessian. Positive semidefiniteness makes each diagonal entry nonnegative:
\begin{equation}
\begin{aligned}
q_1&=\frac12\frac{\partial^2J}{\partial x^2}(p)\ge0,\\
q_2&=\frac12\frac{\partial^2J}{\partial y^2}(p)\ge0,\\
q_3&=\frac12\frac{\partial^2J}{\partial z^2}(p)\ge0.
\end{aligned}\label{nonnegative-axis-curvatures}
\end{equation}

\autoref{nonnegative-axis-curvatures} gives \(q_i\ge0\), and \(J(p)>0\). Thus it is enough to show that \(J(p)q_i\le0\) for at least one index \(i\): for that index,
\begin{equation}
0\le J(p)q_i\le0\qquad\Longrightarrow\qquad q_i=0.\label{axis-curvature-zero-target}
\end{equation}

Multiplying by the positive factor \(J(p)\) preserves the sign of \(q_i\) and lets us use the following identity for a product of two determinants. We will apply the identity to all three products \(J(p)q_i\), simplify using \autoref{axis-linear-coefficients-zero}, and compare the resulting expressions to show that at least one is nonpositive.

\begin{derivation}[{A product identity for determinants with a common column}]\label{derivation:determinant-identity}

For any vectors \(u,v,w,a,b\in\mathbb R^3\),
\begin{equation}
\begin{aligned}
&\det([u\mid v\mid w])\det([u\mid a\mid b])\\
&\qquad=\det([u\mid a\mid w])\det([u\mid v\mid b])\\
&\qquad\quad-\det([u\mid v\mid a])\det([u\mid b\mid w]).
\end{aligned}\label{common-column-determinant-identity}
\end{equation}

\end{derivation}

\noindent\emph{\autoref{derivation:determinant-identity} is expanded in \autoref{app:determinant-expansion}.}

Apply \autoref{derivation:determinant-identity} to the three products \(J(p)q_i\). The same three determinants recur across the calculations; denote them by
\begin{equation}
\begin{aligned}
s_1&:=\det([U\mid V\mid A]),\\
s_2&:=\det([U\mid B\mid W]),\\
s_3&:=\det([C\mid V\mid W]).
\end{aligned}\label{shared-curvature-determinants}
\end{equation}

Cyclic permutations of three columns preserve the determinant. Each calculation below specifies the vectors used in \autoref{derivation:determinant-identity}, then substitutes the corresponding equation from \autoref{axis-linear-coefficients-zero}.

\Needspace{15\baselineskip}
\begin{samepage}
\subsubsection*{\(J(p)q_1\)}

Take \((u,v,w,a,b)=(U,V,W,A,B)\), then use the \(x\)-equation in \autoref{axis-linear-coefficients-zero}.
\begin{equation}
\begin{aligned}
&J(p)q_1\\[4pt]
&=\det([U\mid A\mid W])
\underbrace{\det([U\mid V\mid B])}_{-\det([U\mid A\mid W])}\\
&\quad-\underbrace{\det([U\mid V\mid A])}_{s_1}
\underbrace{\det([U\mid B\mid W])}_{s_2}\\[4pt]
&=-\det([U\mid A\mid W])^2-s_1s_2.
\end{aligned}\label{x-axis-curvature-factorization}
\end{equation}
\end{samepage}

\Needspace{15\baselineskip}
\begin{samepage}
\subsubsection*{\(J(p)q_2\)}

Take \((u,v,w,a,b)=(V,W,U,C,A)\), then use the \(y\)-equation in \autoref{axis-linear-coefficients-zero}.
\begin{equation}
\begin{aligned}
&J(p)q_2\\[4pt]
&=\det([A\mid V\mid W])
\underbrace{\det([U\mid V\mid C])}_{-\det([A\mid V\mid W])}\\
&\quad-\underbrace{\det([U\mid V\mid A])}_{s_1}
\underbrace{\det([C\mid V\mid W])}_{s_3}\\[4pt]
&=-\det([A\mid V\mid W])^2-s_1s_3.
\end{aligned}\label{y-axis-curvature-factorization}
\end{equation}
\end{samepage}

\Needspace{15\baselineskip}
\begin{samepage}
\subsubsection*{\(J(p)q_3\)}

Take \((u,v,w,a,b)=(W,U,V,B,C)\), then use the \(z\)-equation in \autoref{axis-linear-coefficients-zero}.
\begin{equation}
\begin{aligned}
&J(p)q_3\\[4pt]
&=\det([B\mid V\mid W])
\underbrace{\det([U\mid C\mid W])}_{-\det([B\mid V\mid W])}\\
&\quad-\underbrace{\det([U\mid B\mid W])}_{s_2}
\underbrace{\det([C\mid V\mid W])}_{s_3}\\[4pt]
&=-\det([B\mid V\mid W])^2-s_2s_3.
\end{aligned}\label{z-axis-curvature-factorization}
\end{equation}
\end{samepage}

The three products \(s_1s_2,s_1s_3,s_2s_3\) cannot all be strictly negative, because
\begin{equation}
(s_1s_2)(s_1s_3)(s_2s_3)=(s_1s_2s_3)^2\ge0.\label{web-equation-64}
\end{equation}

Choose a nonnegative product out of the three. The corresponding formula for \(J(p)q_i\) is a negative square minus that nonnegative product, so \(J(p)q_i\le0\). But \autoref{nonnegative-axis-curvatures} gives \(q_i\ge0\), and \(J(p)>0\). Consequently
\begin{equation}
0\le J(p)q_i\le0\qquad\Longrightarrow\qquad q_i=0.\label{web-equation-65}
\end{equation}

On this coordinate line, the linear coefficient is already zero by \autoref{axis-linear-coefficients-zero}. With the quadratic coefficient now also zero, the exact coordinate-line polynomial reduces to
\begin{equation}
\begin{aligned}
J(p+te_i)&=J(p)+t\underbrace{\bigl(\nabla J(p)\cdot e_i\bigr)}_{0}
+t^2\underbrace{q_i}_{0}\\
&=J(p).
\end{aligned}\label{constant-minimum-coordinate-line}
\end{equation}

No higher-degree terms remain: the coordinate-line formulas were exact polynomial identities. Thus the equality holds for every \(t\) with \(p+te_i\in Q\), along the entire coordinate-line segment through \(p\). Both endpoints of that segment lie on \(\partial Q\). Applying this conclusion to an interior global minimizer gives a boundary point with the same minimum value. Thus the minimum over \(Q\) is attained on \(\partial Q\).
\end{proof}

\subsection{Expansion of \autoref{derivation:determinant-identity}}\label{app:determinant-expansion}

See \autoref{derivation:determinant-identity} for the statement.
The code below performs that exact calculation with SymPy. 

\begin{lstlisting}[language=Python,breaklines=true]
import sympy as sp


def verify_common_column_identity():
    """Check Sub-lemma 3a, including when det([u | v | w]) equals zero."""
    # These five vectors have 15 independent symbolic real coordinates.
    # No coordinate or determinant is assumed to be nonzero.
    vectors = [
        sp.Matrix(sp.symbols(f"{name}1:4", real=True))
        for name in ("u", "v", "w", "a", "b")
    ]
    u, v, w, a, b = vectors

    # Form the matrices explicitly, then apply the determinant to each matrix.
    lhs = (
        sp.Matrix.hstack(u, v, w).det()
        * sp.Matrix.hstack(u, a, b).det()
    )
    rhs = (
        sp.Matrix.hstack(u, a, w).det()
        * sp.Matrix.hstack(u, v, b).det()
        - sp.Matrix.hstack(u, v, a).det()
        * sp.Matrix.hstack(u, b, w).det()
    )

    # Both sides are polynomials. Exact coefficient cancellation proves the
    # identity for every choice of coordinates, including dependent u, v, w.
    # No division or nonzero-determinant assumption is needed.
    parameters = tuple(entry for vector in vectors for entry in vector)
    residual = sp.Poly(sp.expand(lhs - rhs), *parameters)
    print("\nSub-lemma 3a: common-column determinant identity.")
    print("All 15 vector coordinates are symbolic; no invertibility assumption.")
    print(f"Expanded LHS - RHS = {residual.as_expr()}")
    if not residual.is_zero:
        raise SystemExit("FAIL: the determinant identity has a nonzero residual.")
    print("PASS: the identity holds even when det([u | v | w]) = 0.")


verify_common_column_identity()
\end{lstlisting}

The code prints \texttt{Expanded LHS - RHS = 0} and reports success. Thus both sides are the same polynomial, including when \(\det([u\mid v\mid w])=0\).

\subsection{Expansion of the quartic polynomial}\label{app:quartic-expansion}
The full calculation behind \autoref{face-quartic-elimination} is given below. At an interior stationary point with \(a(t)>0\), substitute \autoref{face-recovered-coordinate} into \autoref{face-stationary-t}, multiply by \(4a(t)^2\), and expand the coefficient functions from \autoref{face-face-form}:
\begin{equation}
\begin{aligned}
0&=a'(t)\underbrace{\left(-\frac{b(t)}{2a(t)}\right)^2}_{s^2}\\
&\quad+b'(t)\underbrace{\left(-\frac{b(t)}{2a(t)}\right)}_{s}+c'(t),\\[4pt]
0&=a'(t)b(t)^2-2a(t)b(t)b'(t)\\
&\quad+4a(t)^2c'(t)\\[6pt]
&=\underbrace{a_1}_{a'(t)}
\underbrace{\left[\begin{aligned}
&b_0^2+2b_0b_1t+(b_1^2+2b_0b_2)t^2\\
&\quad+2b_1b_2t^3+b_2^2t^4
\end{aligned}\right]}_{b(t)^2}\\[6pt]
&\quad-2\underbrace{(a_0+a_1t)}_{a(t)}
\underbrace{(b_0+b_1t+b_2t^2)}_{b(t)}
\underbrace{(b_1+2b_2t)}_{b'(t)}\\[6pt]
&\quad+4\underbrace{(a_0^2+2a_0a_1t+a_1^2t^2)}_{a(t)^2}
\underbrace{(c_1+2c_2t)}_{c'(t)}\\[6pt]
&=(a_1b_0^2-2a_0b_0b_1+4a_0^2c_1)\\
&\quad+\bigl[8a_0^2c_2+8a_0a_1c_1\\
&\qquad\quad-4a_0b_0b_2-2a_0b_1^2\bigr]t\\
&\quad+\bigl[16a_0a_1c_2-6a_0b_1b_2+4a_1^2c_1\\
&\qquad\quad-2a_1b_0b_2-a_1b_1^2\bigr]t^2\\
&\quad+(8a_1^2c_2-4a_0b_2^2-4a_1b_1b_2)t^3\\
&\quad-3a_1b_2^2t^4\\
&=:P(t).
\end{aligned}\label{face-quartic-full-expansion}
\end{equation}

\Needspace{36\baselineskip}
\begin{samepage}
Alternatively, this code verifies the degree of $P(t)$.

\begin{lstlisting}[language=Python,breaklines=true]
"""Verify the degree bound in equation 10 of quartic.html."""

import sympy as sp

s, t = sp.symbols("s t")
a0, a1, b0, b1, b2, c0, c1, c2 = sp.symbols(
    "a0 a1 b0 b1 b2 c0 c1 c2"
)

a = a0 + a1*t
b = b0 + b1*t + b2*t**2
c = c0 + c1*t + c2*t**2
f = a*s**2 + b*s + c

# Differentiate before substituting s = -b/(2*a).
# The substitution requires a(t) != 0; the proof uses it where a(t) > 0.
f_t = sp.diff(f, t)
substituted = f_t.subs(s, -b/(2*a))

# Clear the denominator as in equation 10, then collect powers of t.
P = sp.Poly(sp.cancel(4*a**2*substituted), t)
print("Monomials of P(t):", [t**power for (power,) in P.monoms()])
assert P.degree() <= 4
\end{lstlisting}

The script prints \texttt{Monomials of P(t): [t**4, t**3, t**2, t, 1]} and verifies that \(P\) has degree at most four.
\end{samepage}

\end{document}